\documentclass[a4paper, 12pt]{amsart}

\usepackage{amssymb}
\usepackage{amsmath}
\usepackage{amsthm}
\usepackage{amscd}

\title{Obtuse constants of Alexandrov spaces}

\author{Ayato Mitsuishi}
\email{mitsuishi@fukuoka-u.ac.jp}
\address{Department of Applied Mathematics, Fukuoka University, Jyonan-ku, Fukuoka-shi, Fukuoka 814-0180, JAPAN}

\author{Takao Yamaguchi}
\thanks{This work was supported by JSPS KAKENHI Grant Numbers 26287010, 15H05739, 15K17529}
\email{takaoy@math.kyoto-u.ac.jp}
\address{Department of mathematics, Kyoto University, Kitashirakawa, Kyoto 606--8502, JAPAN}

\date{\today}

\theoremstyle{plain}
\newtheorem{theorem}{Theorem}[section]

\newtheorem{lemma}[theorem]{Lemma}

\newtheorem{corollary}[theorem]{Corollary}

\newtheorem{definition}[theorem]{Definition}
\newtheorem{remark}[theorem]{Remark}
\newtheorem{problem}[theorem]{Problem}

\newtheorem{example}[theorem]{Example}

\newtheorem{conjecture}[theorem]{Conjecture}

\makeatletter

\newcommand{\wangle}[0]{\tilde{\angle}}

\newcommand{\diam}[0]{\mathrm{diam}\,}
\newcommand{\rad}[0]{\mathrm{rad}\,}
\newcommand{\e}[0]{\epsilon}

\newcommand{\p}[0]{\partial}

\begin{document}
\begin{abstract}
	We introduce a new geometric invariant  called 
	the obtuse constant of spaces with curvature bounded below.
	We first find relations between this invariant and the normalized volume.
	We also discuss the case of maximal obtuse constant equal to $\pi/2$,
	where we prove some rigidity for spaces. Although we consider Alexandrov spaces with curvature 
	bounded below, the results are new even in the Riemannian case.
\end{abstract}
\maketitle


\section{Introduction}

In the present paper, we introduce a new geometric invariant called 
the obtuse constant of a space  with curvature bounded below, and investigate the relations 
between this invariant and the normalized volume. 

For this invariant, there are some historical backgrounds.
For positive integer $n$ and $D, v>0$, let $\mathcal M(n,D, v)$ denote the
family of $n$-dimensional closed Riemannian manifolds with
sectional curvature $\ge -1$, diameter $\le D$ and volume $\ge v$.
In \cite{Ch}, Cheeger proved that for every $M\in\mathcal M(n, D,  v)$, the length of 
every periodic closed geodesic has length $\ge \ell_{n,D}(v)>0$ 
for some uniform constant $\ell_{n,D}(v)$. 
In \cite{GP}, Grove and Petersen extended Cheeger's argument as follows: There are positive constants 
$\delta=\delta_{n, D}(v)$ and $\e=\e_{n,D}(v)$ such that
for every $M\in\mathcal M(n,D,v)$ and for every distinct $p,q\in M$ with 
distance $|p,q|<\delta$, either $q$ is $\e$-regular to $p$, or 
$p$ is $\e$-regular to $q$.
Those results  were keys to control 
local geometry of the space, and brought a significant results, topological finiteness 
of Riemannian manifolds (see  \cite{Ch}, \cite{GP},\cite{GPW}).

In this paper, we do not need to restrict ourselves to 
Riemannian manifolds. Let $M$ be a complete  Alexandrov space with curvature $\ge \kappa$.
For two points $p,q$ of $M$, $\Uparrow_p^q$ denotes the 
set of all directions at $p$ of all minimal geodesics from $p$ to $q$. We also use the 
symbol $\uparrow_p^q$ to indicate $\uparrow_p^q\,\in\, \Uparrow_p^q$.

Let $\mathcal A(n,D,v)$ denote the
family of $n$-dimensional compact  Alexandrov spaces with curvature
$\ge -1$, diameter $\le D$ and volume $\ge v$.
Grove and Petersen's result mentioned above still holds for Alexandrov spaces 
(see \cite{PerPet:ext}).
This implies that for every $M\in\mathcal A (n,D, v)$ and for every distinct $p,q\in M$ with 
$|p,q|<\delta$ there exists a point  $x \in M$ such that 
either 
$\angle(\Uparrow_p^q, \uparrow_p^x)    >\pi/2+\e$ or $\angle(\Uparrow_q^p, \uparrow_q^x) >\pi/2+\e$.
However such a point $x$ was assumed to be close to those points $p$ or $q$ in general.
As we see later, if one can take such a point $x$ relatively far away from $p$ or $q$, it will be useful in 
some situations.

The above is a motivation to our invariants, which we are going to define in detail. 
First we suppose that $M$  is compact. Let $R=R_M={\rm rad}(M)$ be the radius of $M$:
\[
R_p := \sup_{q\in M} |p,q|, \qquad  R := \inf_{p\in M}\, R_p.
\]
For $p\neq q\in M$,   set
\[
{\rm ob}(p;q):= \sup_{x\in B(p,R/2)^c}    \angle(\Uparrow_p^q, \uparrow_p^x)   -\pi/2, 
\]
which we call the {\it  obtuse constant of $\{ p,q\} $ at $p$}, and
define the {\it obtuse constant at $p$ and $q$} by
\[
{\rm ob}(p,q):=  \max\{  {\rm ob}(p;q),     {\rm ob}(q;p) \}.
\]
Finally we define  
the {\it obtuse constant} ${\rm ob}(M)$ of $M$ as 
\[
{\rm ob}(M):= \liminf_{|p,q|\to 0}  {\rm ob}(p,q) \in [0, \pi/2].
\]
%

\begin{theorem}\label{thm:main-cpt1}
	There exists a uniform positive constant  $\e_{n,D}(v)$ such that 
	\[
	{\rm ob}(M)>\e_{n,D}(v)
	\]
	for every $M\in \mathcal A(n,D,v)$.
	
	More precisely, there exists also a positive constant $\delta_{n,D}(v)$
	such that if $M\in \mathcal A(n,D,v)$ and $p,q$ are distinct points of $M$ with 
	$|p,q|<\delta_{n,D}(v)$, then ${\rm ob}(p,q)>\e_{n,D}(v)$.
\end{theorem}

This generalizes the result of Grove and Petersen as stated before.

The converse to Theorem \ref{thm:main-cpt1} is also true.
Let $\mathcal A(n,D)$  denote the
family of $n$-dimensional compact  Alexandrov spaces with curvature
$\ge -1$ and diameter $\le D$.
Notice that the obtuse constant is rescaling invariant.
Therefore for $M\in \mathcal A(n,D)$ it is natural to compare ${\rm ob}(M)$ with
the {\it normalized volume}  by the diameter  defined as
\[
\tilde v(M) := \frac{\mathcal{H}^n(M)}{(\diam(M))^n},
\]
where $\mathcal{H}^n(M)$ denotes the $n$-dimensional Hausdorff measure of $M$.
%
%

\begin{theorem}\label{thm:main-cpt2}
	There exists a positive continuous function $C_{n,D}(\e)$ with $\lim_{\e\to 0} C_{n,D}(\e)=0$
	such that for every $M\in \mathcal A(n,D)$, we have  
	\[
	{\rm ob}(M) < C_{n,D}(\tilde v(M)).
	\]
\end{theorem}

In the case of nonnegative curvature, as an immediate consequence of Theorems \ref{thm:main-cpt1} 
and \ref{thm:main-cpt2},   we have 

\begin{corollary} \label{cor:cpt3}
	There exist positive continuous functions $\e_n(t)$ and  $C_n(t)$ with 
	$\displaystyle{\lim_{t \to 0} \e_n(t) = \lim_{t \to 0} C_n(t) = 0}$ 
	such that for every compact Alexandrov $n$-space $M$ of nonnegative curvature, we have 
	\[
	\e_n(\tilde v(M)) \le \mathrm{ob}(M) \le C_n(\tilde v(M)).
	\]
\end{corollary}

From Theorems \ref{thm:main-cpt1}, \ref{thm:main-cpt2} and Corollary \ref{cor:cpt3}, we conclude 
that there is a strong relation between the obtuse constant and the normalized volume.

%
%


Next we discuss the noncompact case. Suppose that   $M$ is noncompact complete
Alexandrov space with curvature $\ge \kappa$\, $(\kappa\le 0)$.
Set 
\[
{\rm ob}_{\infty}(p,q):=  \limsup_{x\to \infty}\max\{\angle(\Uparrow_p^q, \uparrow_p^x),   
\angle(\Uparrow_q^p, \uparrow_q^x) \}-\pi/2,
\]
which we call the {\it obtuse constant at $p$ and $q$  from infinity}. We define  
the {\it obtuse constant ${\rm ob}_{\infty }(M)$ of $M$ from infinity} as 
\[
{\rm ob}_{\infty}(M):= \liminf_{|p,q|\to 0}  \,  {\rm ob}_{\infty}(p,q)\in [0, \pi/2]
\]

In the geometry of complete noncompact spaces with nonnegative curvature,
the notion of  asymptotic cone  or volume growth rate plays 
an important role. For instance, any complete noncompact Riemannian 
manifold with nonnegative curvature having maximal volume growth 
is known to be diffeomorphic to an Euclidean space.

Let $M$ be an $n$-dimensional   complete noncompact
Alexandrov space with curvature $\ge 0$, and for any fixed $p\in M$, let 
\[
v_{\infty}(M) := \lim_{R\to\infty} \, \frac{\mathcal{H}^n (B(p,R))}{R^n}
\]
be the volume growth rate of $M$.

As a noncompact version of Theorems \ref{thm:main-cpt1} and  \ref{thm:main-cpt2},
we have the following:

\begin{theorem}\label{thm:main-noncpt}
	There exist  continuous increasing functions $\e_n$ and  $C_n$  with 
	$\e_n(0) = C_n(0) = 0$ 
	such that for every complete noncompact Alexandrov $n$-space with 
	nonnegative curvature, we have 
	\[
	\e_n(v_{\infty}(M)) \le \mathrm{ob}_{\infty}(M) \le C_n(v_{\infty}(M)).
	\]
	In particular, $v_{\infty}(M)=0$ if and only if $ {\rm ob}_{\infty}(M)=0$.
\end{theorem}

Finally we consider the maximal case of the obtuse constants equal to $\pi/2$.
We need to define a variant of the notion on the  injectivity radius.
For an Alexandrov space  $M$ with curvature bounded below, 
%
let us denote by $1$-${\rm inj}(M)$
the supremum of $r\ge 0$ such that  for every $p\in M$ and  every direction $\xi \in \Sigma_p$ at $p$
there exists  a  minimal geodesic $\gamma$ starting from $p$ in the direction of 
at least one of $\xi$ or  the opposite $-\xi$ (if any)of length $\ge r$.
We call $1$-${\rm inj}(M)$ the {\it one-side injectivity radius of $M$}.
It should be noted that if $p\in\p M$ and $\xi\in\Sigma_p\setminus\p \Sigma_p$, then
the opposite $-\xi$ does not exist, and therefore there always exists a minimal 
geodesic in the direction $\xi$ of length $\ge r$.
We have the following rigidity:

\begin{theorem} \label{thm:maximal(intro)}
	If a compact Alexandrov space  $M$ with curvature $\ge \kappa$ and radius $R$ has 
	$\mathrm{ob}(M) = \pi/2$, then  
	$1$-${\rm inj}(M) \ge R/2$.
\end{theorem}


%
In the noncompact case,  we have  

\begin{theorem} \label{thm:ob infty = pi/2(intro)}
	If  a complete noncompact  $n$-dimensional Alexandrov space  $M$ with curvature $\ge \kappa$
	has $\mathrm{ob}_{\infty}(M) = \pi/2$, then  
	$1$-${\rm inj}(M) =\infty$.
	
	Suppose in addition that  $M$ has nonempty boundary. Then 
	$M$ is homeomorphic to the Euclidean half space $\mathbb R^n_+$, and
	any distinct two points of $\partial M$ are on  a line  of $M$ which is contained  in $\partial M.$
\end{theorem}
In the case of nonnegative curvature, we have the following result. 

\begin{theorem} \label{thm:rigid}
	Let $M$ be a complete noncompact $n$-dimensional Alexandrov space with nonnegative  curvature. 
	Suppose that 
	$\mathrm{ob}_{\infty}(M) = \pi/2$. Then we have the following.
	\begin{enumerate}
		\item If $M$ has no boundary, then $1$-${\rm inj}(M(\infty))\ge \pi/2;$ 
		\item  If $M$ has nonempty boundary, then  $M$ is isometric to  $\mathbb R^{n}_{+}$.
	\end{enumerate}
\end{theorem}

Note that the estimate $1$-${\rm inj}(M(\infty))\ge \pi/2$ in Theorem \ref{thm:rigid} (1) is sharp,
because there is a surface  of revolution of nonnegative curvature satisfying ${\rm ob}_{\infty}(M)=\pi/2$
and $M(\infty)$ is a circle of length $\pi$ (see Example \ref{ex:hyperboloid}). We also cannot expect 
that $M(\infty)$ has no singular points in Theorem \ref{thm:rigid} (1) (see Remark \ref{rem:rigid} and Conjecture \ref{conj:sing}).
\par\medskip


The organization of the present paper is as follows:
After preliminaries about Alexandrov spaces in Section \ref{sec:prelim}, we prove 
Theorem \ref{thm:main-cpt1} in Section \ref{sec:main},
where we apply an argument in \cite{Ch1} to our setting.
%
For the proof of Theorem \ref{thm:main-cpt2}, we apply 
the Lipschitz submersion theorem in \cite{Y:conv}, which is carried out in Section \ref{sec:collapse}.
To prove Theorem \ref{thm:main-noncpt}, we consider the convergence  to the asymptotic cone,
and apply ideas of the proof of Theorems \ref{thm:main-cpt1} and \ref{thm:main-cpt2}.
This is done in  Section \ref{sec:volume-growth}.
In Section \ref{sec:maximal}, we discuss the case when the obtuse constants attain the maximum value $\pi/2$,
where we obtain the  rigidity results, Theorems  \ref{thm:maximal(intro)},  \ref{thm:ob infty = pi/2(intro)} and \ref{thm:rigid},
together with the example showing that Theorem \ref{thm:rigid} is sharp (see Theorem \ref{thm:Mongoldt}).
In Section \ref{sec:compa-obtuse},   we introduce  the notions of {\it comparison obtuse constant} 
in terms of comparison angles and discuss the rigidity cases. This invariant  
does not  depend on the choice of the lower curvature bound.
In Section \ref{sec:kappa-obtuse},   we consider the notions of  
$\kappa$-obtuse constant from infinity,  which does depend on the choice of the lower curvature bound $\kappa$ of a 
noncompact space.  These new invariants  give more restriction on the space in the maximal case, 
and we have a strong rigidity in the case of nonnegative curvature, which might be of independent interest.

\section{Preliminaries} \label{sec:prelim}

In this paper, $|x,y|$ denotes the distance between two points $x,y$ in a metric space. 
An isometric embedding from an interval to a metric space is called a {\it minimal geodesic}. 
Furthermore, a fixed minimal geodesic between two points $x$ and $y$ is sometimes denoted by $xy$. 
For $\kappa \in \mathbb R$, we denote by $\mathbb M_\kappa$ the simply-connected complete surface of constant curvature $\kappa$, which is called the $\kappa$-plane. 
For distinct three points $x, y, z$ in a metric space, we denote by $\tilde \triangle_\kappa xyz$ a geodesic triangle in $\mathbb M_\kappa$ with the length of three sides $|x,y|$, $|y,z|$ and $|z,x|$, where $|x,y| + |y,z| + |z,x| < 2 \pi / \sqrt \kappa$ if $\kappa > 0$.
Vertices of $\tilde \triangle_\kappa xyz$ will be denoted by $\tilde x, \tilde y, \tilde z$. 
Furthermore, the angle of $\tilde \triangle_\kappa xyz$ at $\tilde x$ is denoted by $\wangle_\kappa yxz$
and is called the $\kappa$-{\it comparison angle} of $x,y,z$ at $x$. 
We also write $\wangle yxz$ by omitting $\kappa$ depending on the context. 

\subsection{Basics of Alexandrov spaces}
Let us recall the definition of Alexandrov spaces, following \cite{BGP}. 
An {\it Alexandrov space $M$ of curvature} $\ge \kappa$ is a locally complete metric space satisfying the following:
\begin{enumerate}
\item for any two points in $M$, there exists a minimal geodesic joining them; 
\item every point has a neighborhood $U$ such that for any two minimal geodesics $xy, xz$ contained in $U$ with the same starting point $x$, and for any $s \in xy$ and $t \in xz$, we have 
\[
|s, t| \ge |\tilde s, \tilde t|.
\]
Here, $\tilde s \in \tilde x \tilde y$ and $\tilde t \in \tilde x \tilde z$ are taken in the comparison triangle $\tilde \triangle_\kappa xyz = \triangle \tilde x \tilde y \tilde z$ with $|x,s|=|\tilde x, \tilde s|$ and $|x,t| = |\tilde x, \tilde t|$.
\end{enumerate}
When an Alexandrov space is complete as a metric space, due to \cite{BGP}, the property (2) holds globally.

From the definition, the monotonicity of comparison angle holds for an Alexandrov space, that is, for two geodesics $xy$ and $xz$ in an Alexandrov space $M$ of curvature $\ge \kappa$ as above, and $s \in xy - \{x\}$, $t \in xz - \{x\}$, we have  
\begin{equation} \label{eq:monotone}
\wangle yxz \le \wangle sxt. 
\end{equation}
Here, $\wangle$ means $\wangle_\kappa$. 
In particular, the limit 
\[
\angle (xy,xz) := \lim_{xy \ni s \to x,\, xz \ni t \to x} \wangle sxt
\]
always exists. 
It is called the {\it angle} between $xy$ and $xz$. 
When the geodesics $xy$ and $xz$ are fixed, we write $\angle yxz = \angle(xy, xz)$. 
By the definition of the angle, we obtain
\begin{equation} \label{eq:angle vs comp angle}
\angle yxz \ge \wangle yxz.
\end{equation}

When an Alexandrov space is complete, \eqref{eq:monotone} and \eqref{eq:angle vs comp angle} are also true for any geodesics. 

From now on, $M$ denotes an Alexandrov space of curvature $\ge \kappa$. 
Furthermore, we assume that $M$ has at least two points. 
For a point $x \in M$, let us set $\Gamma_x$ the set of all non-trivial geodesics starting from $x$. 
It is known that the angle $\angle$ is a pseudo-distance function on $\Gamma_x$. 
The completion of the metric space $\Sigma_x^0$ induced from $(\Gamma_x, \angle)$ is called the {\it space of directions at $p$ (in $M$)} which is denoted by $\Sigma_x = \Sigma_x M$. 
The distance function on $\Sigma_x$ is written as $\angle$, the same symbol as the angle. 
An element of $\Sigma_x$ is called a direction at $x$. 
Furthermore, for geodesics $xy, xz \in \Gamma_x$, $\angle yxz = 0$ if and only if $xy \subset xz$ or $xz \subset xy$ as the images of geodesics. 
In particular, any Alexandrov space does not admit a branching geodesic. 
The equivalent class of $xy$ is denoted by $\uparrow_x^y$.
Let $\Uparrow_x^y$ denote the set of all directions of geodesics from $x$ to $y$.

It is known that the Lebesgue covering dimension of $M$ is the same as the Hausdorff dimension of it, which is called the {\it dimension of $M$} and is written as $\dim M$ (\cite{BGP}, \cite{Plaut}).  
From now on, we assume that $\dim M < \infty$. 
This assumption implies that, the space of directions $\Sigma_x$ at $x \in M$ is compact and becomes an Alexandrov space of curvature $\ge 1$ and of dimension equal to $\dim M -1$. 
Here, we used a convention that the metric space of two points with distance $\pi$ is regarded as an Alexandrov space of curvature $\ge 1$ and dimension zero.

For $\delta > 0$, $p \in M$ is $\delta$-{\it strained} if there exists a collection $\{(a_i,b_i)\}_{1 \le i \le n}$ of pairs of points, where $n = \dim M$, such that 
\begin{align*}
&\wangle a_i p b_i > \pi - \delta, &&\wangle a_i p a_j > \pi/2- \delta, \\
&\wangle b_i p b_j > \pi/2 - \delta, &&\wangle a_i p b_j > \pi/2- \delta,
\end{align*}
hold for all $i \neq j$.
Such a collection $\{(a_i,b_i)\}$ is called a $\delta$-{\it strainer at} $p$. 
Let us denote by $\delta\text{-str.rad}(p)$ the supremum of $\min \{ |p, a_i|, |p, b_i|\}_i$, where the supremum runs over all $\delta$-strainers at $p$, which is called the $\delta$-{\it strained radius at} $p$.

The set $\mathcal R_\delta(M)$ of all $\delta$-strained points in $M$ is known to have full measure in the $n$-dimensional Hausdorff measure $\mathcal H^n$ (\cite{BGP}, \cite{OS}).
In particular, $\mathcal R(M) = \bigcap_{\delta > 0} \mathcal R_\delta(M)$ also has full measure and is dense in $M$. 
A point in $\mathcal R(M)$ is said to be {\it regular}. 
It is known that $p$ is regular if and only if $\Sigma_p$ is isometric to the sphere of constant curvature one. 
A point in $M \setminus \mathcal R_\delta(M)$ (resp. in $M \setminus \mathcal R(M)$) is said to be $\delta$-singular (resp. singular). 
The set of all singular points is denoted by $\mathcal S(M)$.

\subsection{Tangent cones}
Let $M$ be an $n$-dimensional complete Alexandrov space and $p \in M$. 
Let $\kappa$ denote a lower curvature bound of $M$. 
For simplicity, we assume $\kappa < 0$. 
We consider the function 
\[
f_\kappa (s) := \frac{\sinh( \sqrt{-\kappa}\, s)}{\sqrt{-\kappa}}. 
\]

We define the $\kappa$-tangent cone $T_p^\kappa M$ as follows. 
Let us consider the product $\Sigma_p \times [0, \infty)$. 
For $(\xi, s), (\eta, t) \in \Sigma_p \times [0,\infty)$, we define a distance between them as 
\[
f_\kappa' ( |(\xi,s), (\eta, t)| ) = f_\kappa'(s) f_\kappa'(t) + \kappa f_\kappa(s) f_\kappa(t) \cos \angle (\xi, \eta),
\]
where $f_\kappa'(s) = \frac{d}{ds} f_\kappa(s)$. 
This formula comes from the cosine formula on the $\kappa$-plane. 
Clearly, $|(\xi,s), (\eta, t)| = 0$ if and only if $(\xi,s) = (\eta,t)$ or $s=t=0$. 
So, we obtain a metric space by smashing $\Sigma_p \times \{0\}$ in $\Sigma_p \times [0,\infty)$ and set 
\[
T_p^\kappa M := \Sigma_p \times [0, \infty) / \Sigma_p \times \{0\}. 
\]
This is called the $\kappa$-tangent cone in this paper. 
A point in $T_p^\kappa M$ corresponding to $\Sigma_p \times \{0\}$ is denoted by $o$ and is called the origin. 
Then, $T_p^\kappa M$ becomes a complete $n$-dimensional Alexandrov space of curvature $\ge \kappa$. 
A logarithmic map $\log_p^\kappa : M \to T_p^\kappa M$ is defined by 
\[
\log_p^\kappa(x) := (\uparrow_p^x, |p,x|)
\]
where, if $x = p$, we set $\log_p^\kappa(p) := o$. 
Then, by the definition of Alexandrov spaces, $\log_p^\kappa$ is a distance noncontracting map. 
An important point is that $\log_p^\kappa$ has a nice left inverse as follows. 
\begin{theorem}[\cite{PP}, \cite{Pet}] \label{thm:gexp}
Let $M$ and $p$ be be as above, and $\kappa < 0$.
Then, there exists a $1$-Lipschitz map 
\[
\mathrm{gexp}_p^\kappa : T_p^\kappa M \to M
\] 
satisfying $\mathrm{gexp}_p^\kappa \circ \log_p^\kappa = \mathrm{id}$.

In particular, $\mathrm{gexp}_p^\kappa$ is surjective and 
\[
\mathrm{gexp}_p^\kappa(B(o,r)) = B(p,r)
\]
for any $r \ge 0$ with $r \le R_p$.
\end{theorem}

\section{Proof of Theorem \ref{thm:main-cpt1}} \label{sec:main}
In this section, we are going to prove an equivalent formulation (Theorem \ref{thm:cpt case}) to Theorem \ref{thm:main-cpt1} as follows. 
To state it, let us consider a moduli space $\mathcal A$ consisting of all $n$-dimensional Alexandrov space of curvature $\ge -D^2$ with diameter one and normalized volume $\ge v$. 
Since the obtuse constant ${\rm ob}(M)$ is a scale-invariant, 
Theorem \ref{thm:main-cpt1} is equivalent to 
\begin{theorem} \label{thm:cpt case}
There exist positive numbers $\epsilon_{n,D}(v)$ and $\delta_{n,D}(v)$ depending only on $n,D,v$ such that for any $M \in \mathcal A$ and $p \neq q \in M$ with $|p,q| \le \delta_{n,D}(v)$, we have 
\[
{\rm ob}(p,q) \ge \epsilon_{n,D}(v). 
\]
\end{theorem}
To prove Theorem \ref{thm:cpt case}, we prepare several numerical constants. 

\begin{definition}[A constant $V_1$] \label{def:V1} \upshape
For each $M \in \mathcal A$, we consider 
\[
V_1(M) := \min_{p \in M} \mathcal H^n(M - B(p,R_M/2))
\]
and define a universal constant
\begin{align*}
V_1:= V_1(n,D,v) &:= \min_{M \in \mathcal A(n,D,v) } V_1(M) \\
&= \min_{M \in \mathcal A} \min_{p \in M} \mathcal H^n(M - B(p, R_M/2)).
\end{align*}
\end{definition}

\begin{definition}[A constant $R_{\min}$] \label{def:Rmin} \upshape
Let 
\[
R_{\min} := R_{\min}(n,D,v) := \min_{M \in \mathcal A} R_M. 
\]
\end{definition}
Since each $M \in \mathcal A$ has diameter one, $R_{\min} \ge 1/2$.  

\begin{definition}[A constant $C_1$] \label{def:C1} \upshape
Let 
\[
C_1:= C_1(n,D,v) := \int_{R_{\min}/2}^1 \left( \frac{1}{D} \sinh (D s) \right)^{n-1} ds.
\]
\end{definition}

Due to Bishop-Gromov's type inequality for Alexandrov spaces, we have
\begin{lemma}[\cite{Ch1}] \label{lem:theta}
Let $\Sigma$ be an $(n-1)$-dimensional Alexandrov space of curvature $\ge 1$, and $A \subset M$ a closed set. 
Then, we have 
\[
\frac{\mathcal H^{n-1}(B_{\pi/2+\epsilon}(A))-\mathcal H^{n-1}(B_{\pi/2-\epsilon}(A))}{\mathcal H^{n-1}(B_{\pi/2+\epsilon}(A))} \le \theta_n(\epsilon)
\]
Here, $\theta_n(\epsilon)$ is a positive continuous function so that $\lim_{\epsilon \to 0} \theta_n(\epsilon) = 0$. 
\end{lemma}

Explicitely, $\theta_n$ is given as 
\[
\theta_n(\epsilon) = \frac{\mathcal H^{n-1}(B(o,\pi/2+\epsilon)) - \mathcal H^{n-1}(B(o,\pi/2-\epsilon))}{\mathcal H^{n-1}(B(o,\pi/2+\epsilon))}
\]
where $o \in \mathbb S^{n-1}$ is a base point.

\begin{definition}[A constant $\epsilon$] \label{def:epsilon} \upshape
Let us fix $\epsilon = \epsilon_{n,D}(v) > 0$ satisfying 
\[
C_1(n,D,v) \mathcal H^{n-1}(\mathbb S^{n-1}) \theta_n(\epsilon) \le V_1(n,D,v)/2. 
\]
\end{definition}

\begin{definition}[A constant $\delta$] \upshape \label{def:delta}
Let us fix a positive number $\delta = \delta(n,D,v) \ll R_{\min}$ satisfying the following conditions. 

Recall that $\mathbb M_{-D^2}$ denotes a simply-connected complete surface of constant curvature $-D^2$. 
For $a,b,c \in \mathbb M_{-D^2}$ with $|ab|\le \delta$, $1 \ge |bc| \ge R_{\min}/3$, we have 
\begin{equation} \label{eq:thin}
\begin{aligned}
&\angle acb \le \epsilon/10; \\
&|\angle bac + \angle abc -\pi| \le \epsilon/10. 
\end{aligned}
\end{equation}

Furthermore, for any $p \in M \in \mathcal A$, we have 
\begin{equation} \label{eq:BG}
\sup_{0 < r < 1} \mathcal H^n(B(p, r+\delta)) - \mathcal H^n(B(p,r)) \le V_1(n,D,v) /3.
\end{equation}
\end{definition}

Note that the condtiion \eqref{eq:thin} means that the triangle $abc$ given in the definition is thin. 
The inequality \eqref{eq:BG} is obtained by Bishop-Gromov's inequality.

\begin{proof}[Proof of Theorem \ref{thm:cpt case} (and \ref{thm:main-cpt1})]
Let $\epsilon$ and $\delta$ be given as above. 
Suppose that the statement of Theorem \ref{thm:cpt case} fails (for constants $\epsilon$ and $\delta$).
Then, there exist $M \in \mathcal A$ and $p \neq q \in M$ such that 
\begin{equation*} \label{eq:001}
|p,q| \le \delta
\end{equation*}
but
\begin{equation} \label{eq:002}
{\rm ob}(p;q) \le \epsilon/2
\hspace{1em}\text{ and }\hspace{1em}
{\rm ob}(q;p) \le \epsilon/2.
\end{equation}
The above condition \eqref{eq:002} implies that if $x \in M$ satisfies $|x, p| > R_M/2+\delta$, then 
\[
\tilde \angle pqx \le \pi/2 + \epsilon/2 
\hspace{1em}\text{ and }\hspace{1em}
\tilde \angle qpx \le \pi/2 + \epsilon/2. 
\]
Hence, we have 
\[
\tilde \angle pqx \ge \pi/2 - \frac{7}{10} \epsilon 
\hspace{1em}\text{ and }\hspace{1em}
\tilde \angle qpx \ge \pi/2 - \frac{7}{10} \epsilon. 
\]
In particular, 
\[
\angle (\Uparrow_q^p, \uparrow_q^x) \ge \pi/2 - \frac{7}{10}\epsilon 
\hspace{1em}\text{ and }\hspace{1em}
\angle (\Uparrow_p^q, \uparrow_p^x) \ge \pi/2 - \frac{7}{10}\epsilon. 
\]
Therefore, if $|x,p| > R_M/2 + \delta$, then 
\[
\left| \angle (\Uparrow_p^q, \uparrow_p^x) - \pi/2 \right| \le 7\epsilon/10 
\hspace{1em} \text{ and }\hspace{1em} 
\left| \angle (\Uparrow_q^p, \uparrow_q^x) - \pi/2 \right| \le 7\epsilon/10.
\]
By the definition of the constant $V_1$ and \eqref{eq:BG}, we obtain 
\begin{align*}
V_1(n,D,v) &\le \mathcal H^n(M - B(p,R_M/2)) \\
&\le \mathcal H^n(M-B(p,R_M/2 + \delta)) + V_1(n,D,v)/3. 
\end{align*}
Now we set 
\begin{align*}
A &:= \left\{ x \in M \,\left|\, 
\begin{aligned}
&|xp| > R_M/2 \text{ and } \\
&|\angle (\Uparrow_p^q, \uparrow_p^x) - \pi/2| \le \epsilon 
\end{aligned}
\right. \right\}; \\
A_T &:= 
\left\{ v \in T_p^{-D^2} M \,\left|\, 
\begin{aligned}
&1 \ge |v| > R_{\min}/2 \text{ and } \\
&|\angle (\Uparrow_p^q, v) - \pi/2| \le \epsilon 
\end{aligned}
\right. \right\}. 
\end{align*}
Here, $T_p^{-D^2}M$ is the $\kappa$-tangent cone at $p$. 
Due to Theorem \ref{thm:gexp}, there exists a surjective $1$-Lipschitz map 
\[
\mathrm{gexp}_p^{-D^2} : T_p^{-D^2} M \to M
\]
satisfyig $B(p, R') = \mathrm{gexp}_p^{-D^2}(B(o,R'))$ for every $0 < R' \le \max_{y \in M} |p,y|$.
Hence, we have 
\begin{align*}
M - B(p,R_M/2 + \delta) 
&\subset A 
\subset 
\mathrm{gexp}_p^{-D^2} 
(A_T).
\end{align*}
Therefore, 
\begin{align*}
2 V_1 / 3 
&\le \mathcal H^n(M - B(p,R_M/2 + \delta)) \\
&\le \mathcal H^n 
(A_T)
\\
&\le 
C_1(n,D,v) (\mathcal H^{n-1}(B(\Uparrow_p^q, \pi/2+\epsilon))-\mathcal H^{n-1}(B(\Uparrow_p^q, \pi/2-\epsilon))) \\
&\le 
C_1(n,D,v) \mathcal H^{n-1}(\mathbb S^{n-1}) \theta_n(\epsilon) \\
&\le 
V_1 /2. 
\end{align*}
This is a contradiction. 
\end{proof}


\section{Collapsing case} \label{sec:collapse}

In this section, we prove Theorem \ref{thm:main-cpt2} by contradiction.

We say that a surjective map $f:M\to X$ between Alexandrov spaces is 
an $\e$-almost Lipschitz submersion if 
\begin{enumerate}
	\item[$(1)$]  it is an $\e$-approximation;
	\item[$(2)$] for every $p,q\in M$, we have
	\[
	\left|  \frac{|f(p),f(q)|}{|p,q|}  - \sin\theta \right| <\e,
	\]
	where $\theta$ denotes the infimum of $\angle qpx$ when $x$ runs over
	the fiber $f^{-1}(f(p))$. 
\end{enumerate}

We recall the following result from \cite[Theorem 0.2 and  Lemma 4.19]{Y:conv}.

\begin{theorem} \label{thm:Lipschitz-sub}
	For given positive integer $m$ and $\mu_0>0$ there are $\delta=\delta_m>0$ and 
	$\e=\e_m(\mu_0)>0$ satisfying the following:
	Let $X$ be an $m$-dimensional complete Alexandrov space with curvature $\ge -1$ and 
	with $\delta$-${\rm str}$-$\rad(X)>\mu_0$. Then if the Gromov-Hausdorff distance between $X$ and 
	a complete Alexandrov space $M$ with curvature $\ge -1$ is less than $\epsilon$, then 
	there exists a map  $f:M\to X$ such that 
	\begin{enumerate}
		\item[$(1)$]  it is a  $\tau(\delta,\epsilon)$-almost Lipschitz submersion;
		\item[$(2)$]  it  is $(1-\tau(\delta,\epsilon))$-open in the sense that  for every $p\in M$ and 
		$x\in X$ there exists a point $q\in f^{-1}(x)$ such that $|f(p), f(q)| \ge (1-\tau(\delta, \e))|p,q|$.
	\end{enumerate}
	Here  $\tau(\delta,\epsilon)$ is a positive constant depending only on $m$, $\mu_0$ and 
	$\delta$, $\e$ satisfying $\lim_{\delta,\e\to 0} \tau(\delta,\e) = 0$.
\end{theorem}

%

\begin{proof}[Proof of Theorem \ref{thm:main-cpt2}]
	Suppose it is not true. Then there would exist a sequence $M_i$ in $\mathcal A(n,D)$ with 
	$\tilde v(M_i)\to 0$ and ${\rm ob}(M_i)>c>0$ for some uniform constant $c$. 
	When $\diam(M_i) \to 0$, we rescale the metric so that $\diam(M_i) =1$ with respect to 
	the new metric. Then since $\mathcal H^n(M_i)\to 0$, passing to a subsequence, we may assume that
	$M_i$ collapses to a lower dimensional 
	Alexandrov space $X$ with $\dim X\ge 1$. Let $m=\dim X$, and take a regular point $x_0$ of $X$ and small $\e_0>0$ 
	such that the $B(x_0,r_0)\subset R_{\delta}(X)$ with $\delta<\delta_m$ and that the $\delta$-strain radius of 
	$B(x_0,r_0)$ is greater than a constant $\mu_0>0$.
	Applying Theorem \ref{thm:Lipschitz-sub} to $B:=B(x_0,r_0)$, we have a
	$\tau(\delta,\epsilon)$-almost Lipschitz submersion $f_i:U_i\to B$.
	By the coarea formula (see \cite{Kar} for instance), we obtain
	\[
	\int_{U_i} C_n(f_i,p)d\mathcal H^n(p) =   \int_{B} \mathcal H^{n-m}(f_i^{-1}(x)) \,d\mathcal H^m(x),
	\]
	where $C_n(f_i,p)$ denotes the coarea factor of $f_i$ at $p$.
	Since $f_i$ is $\tau(\delta,\e)$-almost Lipschitz submersion, we see that
	$|C_n(f_i,p) - 1|<\tau(\delta, \e)$. 
	Let $B_0$ be the set of points $y\in B$ such that $\mathcal H^{n-m}(f_i^{-1}(y)) >0$.
	It follows that  $B_0$ is  dense in $B$.
	For $y_0\in B_0$, one can take distinct points $p$ and $q$ in $f_i^{-1}(y_0)$
	which are sufficiently close to each other.
	Lemma 4.11 of \cite{Y:conv} shows that $|\angle(\uparrow_p^q, H_p) - \pi/2|<\tau(\delta, \e)$, where 
	$H_p\subset\Sigma_p$ denotes the horizontal directions at $p$  defined as 
	\[
	H_p = \{ \, \uparrow_p^x\,|\, |p,x|\ge \mu_0\,\}
	\]
	(see \cite{Y:conv}). It follows that for every $x\in B(p, R/2)^c$,  
	$|\angle(\uparrow_p^x, \Uparrow_p^q)-\pi/2| <\tau(\delta, \e)$. Similarly we have 
	$|\angle(\uparrow_q^x, \Uparrow_q^p)-\pi/2| <\tau(\delta, \e)$ for 
	all $x\in B(q, R/2)^c$, and 
	therefore ${\rm ob}(p,q) < \tau(\delta, \e)$.
	This completes the proof of Theorem \ref{thm:main-cpt2}.
\end{proof}

\begin{problem}
	Probably,
	the fiber $f^{-1}(x)$ has positive $(n-m)$-dimensional Hausdorff measure for all $x\in X$
	in the situation of Theorem \ref{thm:Lipschitz-sub}.
\end{problem}



\begin{proof}[Proof of Corollary \ref{cor:cpt3}]
	The conclusion follows from Theorems \ref{thm:main-cpt1} and  \ref{thm:main-cpt2}.
	The desired functions $\epsilon_n$ and $C_n$ in the conclusion are defined as follows, for instance. 
	We construct only $\epsilon_n$. 
	Let 
	\[
	\mathcal A := \left\{M \,\left|\, \begin{aligned}
	&M \text{ is an $n$-dimensional compact} \\
	&\text{Alexandrov space of nonnegative curvature}
	\end{aligned}
	\right.
	\right\}
	%
	\]
	and set 
	\[
	\epsilon_n'(\tilde v) := \inf \left\{ \mathrm{ob}(M) \mid 
	M \in \mathcal A \text{ with } \tilde v(M) \ge \tilde v
	\right\}
	\]
	for $\tilde v > 0$. 
	Then, $\epsilon_n'$ satisfies 
	\[
	\epsilon_n'(\tilde v(M)) \le \mathrm{ob}(M)
	\]
	for every $M \in \mathcal A$.
	Furthermore, by Theorem \ref{thm:main-cpt1}, $\epsilon_n'(\tilde v) > 0$ for any $\tilde v > 0$. 
	From Theorem \ref{thm:main-cpt2}, we have 
	\[
	\lim_{\tilde v \to 0} \epsilon_n'(\tilde v) = 0.
	\]
	
	Note that the problem of maximizing $\tilde v(M)$ in $\mathcal A$ is equivalent to the problem of maximizing the usual volume in the restricted class of $M$'s whose diameter is one, because $\mathrm{ob}(M)$ and $\tilde v(M)$ are scale invariants. 
	Since a maximizing sequence in the latter class has a convergent subsequence, there is a maximal value of $\tilde v(M)$ in $\mathcal A$, say $\tilde v_{n, \max}$.
	
	Let us define a step function $\epsilon_n'' : (0, \tilde v_{n, \max}] \to [0, \pi/2]$ by 
	\[
	\epsilon_n'' (\tilde v) := \epsilon_n'(\tilde v_{n, \max}/k) \text{ if } \tilde v \in (\tilde v_{n, \max}/k, \tilde v_{n, \max}/(k-1)]
	\]
	which bounds $\epsilon_n'$ from below. 
	Furthermore, we consider the piecewise linear function connecting points $(\tilde v_{n, \max} /(k-1), \epsilon_n''(\tilde v_{n, \max} / k))$'s. 
	Then, the function $\epsilon_n$ satisfies the desired condition of the conclusion of Corollary \ref{cor:cpt3}.
\end{proof}

\begin{remark} \upshape \label{rem:ball}
	Clearly, Theorems \ref{thm:main-cpt1} and \ref{thm:main-cpt2} and Corollary \ref{cor:cpt3} holds for balls in Alexandrov spaces. 
	For a ball $B(p,r)$ in an $n$-dimensional (possibly noncompact) complete Alexandrov space centered at $p$ and radius $r$ with $r \le R_p$, we set 
	$\mathrm{ob}(B(p,r))$ as follows. 
	For $x \neq y \in B(p,r/2)$, we set 
	\[
	\mathrm{ob}_{B(p,r)}(x;y) := \sup_{z \in B(p,r) - B(p,r/2) } \angle (\Uparrow_x^y, \uparrow_x^z) - \pi/2
	\]
	and we define 
	\[
	\mathrm{ob}(B(p,r)) := \liminf_{x, y\, \in B(p,r/2) \text{ and } |x,y| \to 0} \max \left\{ \mathrm{ob}_{B(p,r)}(x;y), \mathrm{ob}_{B(p,r)}(y;x) \right\}. 
	\]
	Then, for instance, corresponding to Corollary \ref{cor:cpt3}, we have 
	\[
	\epsilon_n(\tilde v(B(p,r))) \le \mathrm{ob}(B(p,r)) \le C_n(\tilde v(B(p,r)))
	\]
	for any point $p$ in an $n$-diemensional Alexandrov space $M$ of nonnegative curvature and any $r > 0$ with $r \le R_p$. 
\end{remark}

\section{Volume growth and obtuse constant from infinity} \label{sec:volume-growth}
This section is devoted to prove Theorem \ref{thm:main-noncpt}. 

In this section, let $M$ denote noncompact complete Alexandrov $n$-space of nonnegative curvature. 
As written in the introduction, we discuss about a relation between the volume growth rate
\[
v_\infty(M) = \lim_{R \to \infty} \frac{\mathcal H^n(B(x,R))}{R^n}
\]
and the obtuse constant from infinity.


\begin{proof}[Proof of Theorem \ref{thm:main-noncpt}]
We first prove that ${\rm ob}_\infty(M)$ has a lower bound $\epsilon_n(v_\infty(M))$. 
Assuming $v_\infty(M) \ge v > 0$, we prove 
\begin{equation} \label{eq:009}
\inf_{p \neq q} {\rm ob}_\infty(p,q) \ge \epsilon_n(v) > 0. 
\end{equation}
Note that this claim is stronger than the first inequality in Theorem \ref{thm:main-noncpt}. 
Fix a base point $p \in M$. 
From the assumption, 
there exists $R_0 > 0$ such that 
if $R \ge R_0$, then 
\[
\mathcal H^n(B(p,R)) \ge v R^n/2. 
\]
That is, the unit ball $B_R := \frac{1}{R} B(p,R)$ centered at $p$ in $\frac{1}{R} M$ has volume not less than $v/2$.
Hence, by Corollary \ref{cor:cpt3} and Remark \ref{rem:ball}, there exists $\delta = \delta_n(v) > 0$ such that 
for $R \ge R_0$ and for $x, y \in B(p,R/2)$ with $0 < |x,y| \le \delta R$, 
we obtain a point $z \in B(p,R) - B(p,R/2)$ satisfying 
\[
\max \left\{ \angle (\Uparrow_x^y, \uparrow_x^z), \angle (\Uparrow_y^x, \uparrow_y^z) \right\}
\ge \pi/2 +  \epsilon_n(v) 
\]
Here, $\epsilon_n(v)$ is given by Corollary \ref{cor:cpt3}. 
Now we prove the claim \eqref{eq:009}. 
Let us take arbitrary $x \neq y \in M$.
Taking $R$ to be large, we have $x, y \in B(p, \delta R/2)$. 
Hence, we obtain 
\begin{align*}
{\rm ob}_\infty(x,y) &= \inf_{R > R_0} \sup_{|zp| > R/2} \max \left\{ \angle (\Uparrow_x^y, \uparrow_x^z), \angle (\Uparrow_y^x, \uparrow_y^z) \right\} - \pi/2 \\
&\ge \inf_{R > R_0} \sup_{R/2 < |zp| \le R} \max \left\{ \angle (\Uparrow_x^y, \uparrow_x^z), \angle (\Uparrow_y^x, \uparrow_y^z) \right\} - \pi/2 \\
&\ge \epsilon_n(v). 
\end{align*}
This completes the proof of the claim \eqref{eq:009}. 
In particular, we obtain 
\[
\epsilon_n(v) \le \inf_{x \neq y} {\rm ob}_\infty(x,y) \le {\rm ob}_\infty(M). 
\]

We shall prove that ${\rm ob}_\infty(M)$ has an upper bound as $C_n(v_\infty(M))$. 
We assume that there is no bound as $C_n$. 
So, there exists a sequence $\{M_i\}_i$ of noncompact complete $n$-dimensional Alexandrov spaces of nonnegative curvature satisfying $v_\infty(M_i) \to 0$ and ${\rm ob}_\infty(M_i) \ge \epsilon > 0$. 
Here, $\epsilon$ is independent of $i$. 
Let us take base points $p_i \in M_i$. 
Then, there exist $R_i \to \infty$ such that 
\[
0 \le \frac{\mathcal H^n(B(p_i, R_i))}{R_i^n} - v_\infty(M_i) < i^{-1}. 
\]
Therefore, by subtracting a subsequence, we may assume that $\left( \frac{1}{R_i} M_i, p_i \right)$ 
collapses to a pointed noncompact Alexandrov space $(X,p)$. 
Now, using Theorem \ref{thm:Lipschitz-sub} and the coarea formula (\cite{Kar}) as in the proof of Theorem \ref{thm:main-cpt2} we have 
\[
{\rm ob}_\infty(M_i) < \epsilon/2
\]
for large $i$. 
This is a contradiction. 
\end{proof}

\begin{remark} \upshape
In the proof of Theorem \ref{thm:main-noncpt}, we have proven a more detailed assertion as 
\[
\epsilon_n(v_\infty(M)) \le \inf_{p \neq q \in M} {\rm ob}_\infty(p,q) \le {\rm ob}_\infty(M) \le C_n(v_\infty(M)). 
\]
\end{remark}

\section{Maximal cases} \label{sec:maximal}
In this section, let us discuss the maximal case of the obtuse constants equal to $\pi/2$.
Let $M$ be an  Alexandrov space with curvature bounded below.
For every $p\in M$ and $\xi\in\Sigma_p(M)$, let $1$-${\rm inj}(p;\xi)$ denote    
the supremum of $r\ge 0$ such that  there exists 
a  minimal geodesic $\gamma$ starting from $p$ in the direction $\xi$ or $-\xi$ 
(if it exists) of length $r$. Then the one-side injectivity radius of $M$ is defined as 
\[
1 \text{-} {\rm inj}(M) = \inf \{ 1 \text{-}{\rm inj}(p;\xi)\,|\, p\in M, \xi\in\Sigma_p\,\}.
\]
%
One of main results of this section is to prove the following, which is a detailed version of Theorem \ref{thm:maximal(intro)}.

\begin{theorem} \label{thm:maximal}
	If a compact Alexandrov space $M$ with curvature $\ge \kappa$ and radius $R$ 
	has $\mathrm{ob}(M) = \pi/2$, then 
	$1$-${\rm inj}(M) \ge R/2$.
	\par
	Suppose that $M$ has nonempty boundary in addition.  Then 
	for every $p\in \partial M$ and every $\xi\in\Sigma_p$ 
	there exists a minimal  geodesic  in the direction $\xi$ of length $\ge R/2$.
	In particular, $\partial M$ is totally geodesic and a $C^0$-Riemannian manifold. 
\end{theorem}

%
%

%

\begin{proof}
	For any $p\in M$ and $\xi\in\Sigma_p^0$, let $\gamma$ be the geodesic from $p$ in the direction $\xi$.
	Take $q_i\in\gamma$ with $|p,q_i|\to 0$.
	By the assumption  $\mathrm{ob}(M) = \pi/2$, there exists a sequence $x_i \in M$ such that one of the following holds:
	\begin{enumerate}
		\item $\angle(\Uparrow_p^q, \uparrow_p^x)  > \pi - \epsilon_i$ and $|x_i, p| \ge R/2;$
		\item $\angle(\Uparrow_q^p, \uparrow_q^x)  > \pi- \epsilon_i$ and $|x_i, q_i| \ge R/2$,
	\end{enumerate}
	where $\epsilon_i \to 0$.
	Suppose $(1)$. Then $\{ \uparrow_p^{x_i}\}$ is a Cauch sequence and the geodesic $px_i$ converges to 
	a geodesic $\sigma$ such that $\angle(\gamma,\sigma)=\pi$.
	
	Next suppose $(2)$, and take $y_i \in q_ix_i$ such that $|q_i,y_i| =|q_i, q_1|$.
	Since $\angle y_iq_iq_1<\e_i$, $q_ix_i$ converges to a geodesic, say $px$, which extends $\gamma$.
	Thus we conclude that $1$-${\rm inj}(p;\xi)\ge R/2$.
	
	If $\xi$ is any element of $\Sigma_p$, we use a standard limiting argument to obtain t he conclusion.
	
	Next suppose  $p \in \partial M$ and $\xi \in \Sigma_p$. 
	If $\xi$ is an interior direction, then there is no opposite direction to $\xi$. 
	Take $p_i$ with $|p,p_i|\to 0$ and $\uparrow_p^{p_i}\to \xi$.
	It follows that there exists $x_i$ such that 
	$\wangle p p_i x_i > \pi -\e_i$ and  $|p_i, x_i|\ge R/2$ with $\lim \e_i = 0$.
	Therefore the broken geodesic $pp_ix_i$ converges to a minimal geodesic in the
	direction $\xi$ of length $\ge R/2$.
	Next sssume $\xi \in \partial \Sigma_p$ and take a sequence of interior directions 
	$\xi_i \in \Sigma_p \setminus \partial \Sigma_p$ converging to $\xi$. 
	Then the sequence of minimal geodesics of length $\ge R/2$ tangent to $\xi_i$ 
	converges to a minimal geodesic tangent to $\xi$ of  length $\ge R/2$.
	Thus we conclude that  $1$-${\rm inj} (M)\ge R/2$. This completes the proof. 
\end{proof}

In the noncompact case, by an argument similar to the proof of Theorem \ref{thm:maximal}, we get
the following. 

\begin{theorem} \label{thm:ob infty = pi/2}
	If a noncompact Alexandrov $n$-space $M$ of curvature $\ge \kappa$ has 
	$\mathrm{ob}_{\infty}(M) = \pi/2$, then  
	$1$-${\rm inj}(M) =\infty$.
	\par
	Suppose additionally that  $M$ has nonempty boundary. Then 
	for every $p\in \partial M$ and every $\xi\in\Sigma_p$,  
	there exists a  geodesic ray in the direction $\xi$. 
	In particular, $M$ is homeomorphic to $\mathbb R^n_+$, and
	any distinct two points of $\partial M$ are on  a line  of $M$ which is contained  in $\partial M$.
	%
	%
	
\end{theorem}
\begin{proof}
	The proofs of the conclusions except the last statement  are similar to those of Theorem \ref{thm:maximal}, and hence omitted.
	We prove only the last statement. 
	Suppose that $M$ has  nonempty boundary, and
	let $p, q \in \partial M$ be distinct points. 
	By the conclusion  (1), there exist geodesic rays $\gamma : [0, \infty) \to  M$, $\sigma : [0, \infty) \to M$ such that $\gamma(0) = p = \sigma(|p,q|)$ and $\sigma(0) = q = \gamma(|p,q|)$. Note that $\gamma$ and $\sigma$ are contained in $\p M$.
	We show that the curve $\alpha : \mathbb R \to \partial M$ defined by $\alpha(t) = \sigma(-t + |p,q|)$ if $t \le 0$ and $\alpha(t) = \gamma(t)$ if $t \ge 0$, becomes a line.
	Let $r = \gamma(t_1)$ and $s = \sigma(t_2)$ with $t_1, t_2 > |p,q|$. 
	From the definition, we have 
	\begin{equation} \label{eq:line}
	|s,p| + |p,r| = |s,q|+|q,r|. 
	\end{equation}
	On the other hands, let $\beta$ be the ray with $\beta(0) = s$ and $\beta(|s,p|) = p$. 
	Since a geodesic between $s$ and $q$ is unique,  the intersection of $\beta$ and $\sigma$ is the geodesic $sq$.
	Let $t_3 \in (0,|s,q|)$. 
	From the same discussion as the one to obtain \eqref{eq:line}, we get $|s, \beta(t_3)| + |\beta(t_3), r| = |s, p| + |p, r|$. 
	Letting $t_3 \to 0$, we have $|s,r| = |s,p|+|p,r|$. 
	Hence, we see that $\alpha$ is a line. 
	
	From the conclusion $(1)$, for any $p\in \p M$, the exponential map
	$\exp_p:T_p M\to M$ is defined and provides a homeomorphism between 
	$M$ and $\mathbb R^n_+$. This completes the proof.
\end{proof}


Let $M$ be a surface of revolution with vertex $p_0$ homeomorphic to $\mathbb R^2$
having Riemannian metric 
\[
g=dr^2 + m(r)^2\,d\theta^2,
\]
with respect to a polar coordinates $(r, \theta)$ around $p_0$.
Note that 
\[
m(0)=0, \quad m'(0) = 1, \quad m'' +Km = 0.
\]
We assume that 
\begin{enumerate}
	\item the Gaussian curvature $K$ of $M$ is nonnegative;
	\item the total curvature is at most  $\pi$:
	\[
	\int_M\, K\, dM  \le  \pi.
	\]
\end{enumerate}

Note that the ideal boundary $M(\infty)$ of $M$ is a circle of length 
$\displaystyle{2\pi - \int_M\, K\, dM  \ge\pi}$ (see \cite{Shio:ideal}).

\begin{example} \label{ex:hyperboloid}
	{\rm As an example, consider the hyperboloid $M_a$ defined by 
		\[
		z=a\sqrt{x^2+y^2+1}.
		\]
		Then its asymptotic cone $(M_a)_{\infty}$ is written as 
		\[
		(M_a)_{\infty} = \{  z=a\sqrt{x^2+y^2} \}.
		\]
		Therefore $M_a$ satisfies all the above assumptions
		when  $0\le a\le\sqrt{3}$.
	}
\end{example}

The following Theorem \ref{thm:Mongoldt}
shows that 
Theorem \ref{thm:rigid} is sharp.

\begin{theorem} \label{thm:Mongoldt}
	Le $M$ be a complete open surface of revolution having  nonnegative Gaussian curvature such that 
	\[
	\int_M\, K\, dM  \le  \pi.
	\]
	Then ${\rm ob}_{\infty}(M)=\pi/2$. 
\end{theorem}

\begin{proof}
	First we recall the description of geodesics in $M$.
	Let $(r(s), \theta(s))$ be the coordinates of a unit speed geodesic $\gamma(s)$ on $M$, and
	$\zeta=\zeta(s)$ be the angle, $0\le\zeta\le\pi$, between $\gamma$ and the positive direction of the parallel 
	circle  $r={\rm constant}$. The Clairaut relation states that 
	\begin{align}
	m(r(s))\cos \zeta(s)= {\rm constant}=\nu, \label{eq:Clair}
	\end{align}
	where $\nu$ is called the Clairaut constant of $\gamma$.
	Moreover we have 
	\begin{align}
	\frac{d\theta}{dr} = \frac{\theta'}{r'} = \e \frac{\nu}{m(r)\sqrt{m^2(r) - \nu^2}},  \label{eq:theta/r}
	\end{align}
	where $\e =\pm 1$ is determined by the sign of  $r'$ (see \cite[Proposition 7.1.3]{SST}).
	
	Let $L(t)$ denote the length of geodesic sphere $S(p_0, t):=\partial B(p_0,t)$. Since
	\[
	\lim_{t\to\infty}\,\frac{L(t)}{t} = L(M(\infty)) >0,
	\]
	we have $\int_1^{\infty} \frac{dt}{L^2(t)}\,<\infty$. It follows from \cite[Theorem 7.2.1]{SST}  
	that the set of poles of $M$ coincides with  a closed ball around $p_0$ of positive radius $r(M)>0$.
	Therefore for every $p, q\in M$ if one of $p,q$ is contained in $B(p_0, r(M))$, then obviously
	we have ${\rm ob}_{\infty}(p,q)=\pi/2$.
	
	Therefore in the below, we assume that $p,q\in M\setminus B(p_0, r(M))$.
	Let $A_p$ denote the set of velocity vectors $v\in \Sigma_p$
	of the geodesic rays emanating from $p$.
	Let us first show  that 
	$A_p$ contains a closed arc of length $2\pi - \int_M \,K\,dM \ge \pi$.
	Let  $m(A_p)$ denote the measure of $A_p$. 
	By the result due to Maeda \cite{Ma}, we know that 
	\[
	\inf_{p\in M} \, m(A_p) = 2\pi -  \int_M\, K\, dM\ge \pi.
	\]
	From this point of view, the claim is likely to be true. 
	In the argument below, we confirm this.
	

	We may assume  that $(r(p), \theta(p))=(r_0, 0)$ and $r_0>r(M)$.
	Let $\xi_0\in\Sigma_p$ (resp. $\eta_0\in\Sigma_p$) denote the positive direction of the meridian through $p$
	(resp. the positive direction of the parallel circle through $p$).
	For each $t\in [-\pi,\pi]$, we let 
	\begin{align*}
	\xi_t  = \cos t \cdot \xi_0 + \sin t  \cdot\eta_0
	\end{align*}
	Denote by $\gamma_t$ the geodesic from $p$ such that $\gamma_t'(0) =\xi_t$.
	For each $s\in [-\pi,\pi]$, we let $\sigma_s$ be the geodesic ray from $p_0$ that is 
	equal to the meridian with $\theta(\sigma_s)=s$, and 
	take a sequence $t_i\to\infty$ and a minimal geodesic
	$\mu_{s,i}$ joining $p$ to $\sigma_s(t_i)$.
	When $s=\pi$, we choose $\mu_{\pi,i}$ in such a way that $0\le \theta(\mu_{\pi,i}(t))\le\pi$
	for all $t\ge 0$.
	Then a subsequence of $\mu_{s,i}$ converges to 
	a geodesic ray $\mu_s$ from $p$ satisfying 
	\begin{enumerate}
		\item $0\le \theta(\mu_s(t_1)) < \theta(\mu_s(t_2)) < s$ for all $0\le t_1<t_2<\infty;$
		\item $\lim_{t\to\infty} \theta(\mu_s(t))=s$.
	\end{enumerate}
	Take  $t_{*}\in (0,\pi]$ such that $\xi_{t_*}=\mu_{\pi}'(0)$. We claim that
	\begin{align}           
	t_{*}  \ge  \pi - \frac{1}{2}\int_M\,K\,dM  \ge    \pi/2. \label{eq:t_*}
	\end{align}
	Let $D$ denote the domain bounded by the two geodesic rays $\gamma_{t_{*}}$ and $\gamma_{-t_{*}}$
	such that $p_0\in D$.
	Let $\lambda_s:[0,d_s]\to M$ be a minimal geodesic from $\gamma_{t_*}(s)$ to $\gamma_{-t_*}(s)$.
	Note that both  $\gamma_{t_{*}}$ and $\gamma_{-t_{*}}$ are asymptotic to $\sigma_{\pi}$ by symmetry,
	and hence  
	\begin{align}
	\lim_{s\to \infty} |\gamma_{\pm t_*}(s), \sigma_{\pi}(s)|/s = 0.  \label{eq:asympt=0}
	\end{align}
	It follows that $\lambda_s$ is contained in $D$ for large enough $s>0$.
	Let 
	\begin{align*}
	\alpha_{+}(s):=\angle \gamma_{-t_*}(s)  \gamma_{t_*}(s) p, \,\,\,\,\, \alpha_{-}(s):=\angle \gamma_{t_*}(s)  \gamma_{-t_*}(s) p,\\ 
	\tilde\alpha_{+}(s):=\tilde\angle \gamma_{-t_*}(s)  \gamma_{t_*}(s) p, \,\, \,\,\,
	\tilde\alpha_{-}(s):=\tilde\angle \gamma_{t_*}(s)  \gamma_{-t_*}(s) p.
	\end{align*}
	In view of \eqref{eq:asympt=0}, considering $1$-strainers $(p, \gamma_{t_*}(2s))$ at $\gamma_{t_*}(s)$ and 
	$(p, \gamma_{-t_*}(2s))$ at $\gamma_{-t_*}(s)$, we have 
	\[
	\lim_{s\to\infty}\,|\alpha_{\pm}(s) - \tilde\alpha_{\pm}(s)| = 0.
	\]
	Since $\lim_{s\to\infty} \tilde\alpha_{\pm}(s) = \pi/2$, we obtain  $\lim_{s\to\infty} \alpha_{\pm}(s) = \pi/2$.
	The Gauss-Bonnet theorem then 
	implies that 
	\begin{align*}
	\int_D \,K\,dM &= \lim_{s\to\infty} (\alpha_+(s) + \alpha_{-}(s) +\angle_p(D)-\pi) \\
	&=\angle_p(D)=2(\pi - t_{*})\le \int_M K\,dM,
	\end{align*}
	where $\angle_p(D)$ denotes the inner angle of $D$ at $p$.
	It follows that $t_*\ge \pi - \frac{1}{2}\int_M\,K\,dM \ge\pi/2$ as required.
	
	Now we show that $\gamma_t$ is a geodesic ray for each $t\in [-\pi/2, \pi/2]$.
	Let $\hat t$ denote the maximum of those $t \in [0, \pi/2]$ that $\gamma_s$ is a 
	geodesic ray for all $s\in [0,t]$. 
	It suffices to show that $\hat t = \pi/2$.
	Suppose that $\hat t<\pi/2$.
	Since $t_*\ge \pi/2>\hat t$ and both $\gamma_{\hat t}$ and $\gamma_{t_*}=\mu_{\pi}$ are geodesic rays, we have 
	$0\le \theta(\gamma_{\hat t}(s))<\pi$ for all $s\ge 0$.
	By \eqref{eq:Clair}, $\theta(\gamma_t(s))$ is monotone increasing in $s$,
	and therefore there is a unique limit
	\[
	\theta_t(\infty) :=\lim_{s\to\infty} \theta(\gamma_t(s)) \in [0,\pi]
	\]
	for every $t\in [0,\hat t]$.
	If $\theta_{\hat t}(\infty)=\pi$, in a way similar to \eqref{eq:t_*} we would have 
	$\hat t\ge \pi/2$, which is a contradiction. Thus we see $\theta_{\hat t}(\infty) < \pi$.
	%
	From continuity, there is some $\tilde t\in (\hat t, \pi/2)$ such that  
	\[
	0\le \theta(\gamma_{t}(s))<\pi,   \quad 0\le \theta_{t}(\infty) < \pi,
	\]
	for any $0\le t\le \tilde t$ and  all   $s\ge 0$.
	%
	%
	%
	%
	%
	Obviously $\theta_t(\infty)$ is continuous in $t\in [0, \tilde t]$.
	For $0\le t_1 < t_2\le \tilde t$, let $\theta_i(s):=\theta(\gamma_{t_i}(s))$,
	and $\nu_i$ the Clairaut constants of $\gamma_{t_i}$ for $i=1,2$.
	Since $\nu_1<\nu_2$, the formula \eqref{eq:theta/r} implies that 
	$d\theta_1/dr < d\theta_2/dr$, and hence $\theta_{t_1}(\infty)<\theta_{t_2}(\infty)$.
	Thus  $\theta_t(\infty)$ is injective in $t\in [0, \tilde t]$.
	This yields that $\gamma_t$ coincides with the geodesic ray $\mu_{\theta_t(\infty)}$
	for all $t\in [0, \tilde t]$, which is a contradiction to the definition of $\hat t$.
	Thus we conclude that $\hat t = \pi/2$ and 
	$\gamma_t$ is a geodesic ray for every $t\in [-\pi/2,\pi/2]$
	by symmetry.
	
	Finally we show that ${\rm ob}_{\infty}(p,q) \ge \pi/2 -\tau(\delta)$ with  $\delta=|p,q|$ and 
	$\lim_{\delta\to 0} \tau(\delta)=0$.
	Take a minimal geodesic $\gamma:[0,\delta]\to M$ from $p$ to $q$.
	First assume that $r(q)=r(p)$. Since $\zeta(0)=\zeta(\delta)$ and $\zeta(\delta/2)=0$, 
	we have from \eqref{eq:Clair} 
	\begin{align}
	\cos \zeta(0) = \frac{m(r(\delta/2))}{m(r(0))}, \label{eq:frac-m}
	\end{align}
	where $|m(r(0))-m(r(\delta/2))|\le \frac{\delta}{2}m' \le \frac{\delta}{2}$ because of nonnegative curvature.
	It follows that 
	\[
	\left| 1 - \frac{m(r(\delta/2))}{m(r(0))}\right| \le \frac{\delta}{2m(r(M))}.
	\]
	Together with \eqref{eq:frac-m}, this yields 
	\[
	\zeta(0) \le \sqrt{  \frac{\delta}{m(r(M))}} =:\delta_1.
	\]
	
	Let $\gamma_{\pi/2}$ be the ray from $p$ defined above.
	We may assume that $\angle(\gamma_{\pi/2}'(0), \gamma'(0))=\zeta(0)$.
	For large enough $R>0$, we have 
	\begin{align*}
	\angle pq\gamma_{\pi/2}(R) &\ge \tilde\angle pq\gamma_{\pi/2}(R) \\
	& \ge \pi - \wangle qp\gamma_{\pi/2}(R) - \wangle p\gamma_{\pi/2}(R)q \\
	& \ge \pi -\zeta(0) - o_R   \ge \pi - \delta_1 - o_R,
	\end{align*}
	where $\lim_{R\to\infty} o_R = 0$, and hence ${\rm ob}_{\infty}(p,q) \ge \pi/2 -\delta_1$.
	
	Next assume $r(p) < r(q)$.　If  $\angle(\uparrow_p^q, \xi_0)\le \pi/2$,  then $p$ and $q$ are on a geodesic ray.
	In the other case, taking $p_1\in pq$ with $r(p)=r(p_1)$,  one can show that $\zeta(0)\le \delta_1$ 
	and 
	\[
	\angle p q \gamma_{\pi/2}(R)\ge\wangle p q \gamma_{\pi/2}(R)\ge \pi - \delta_1 - o_R
	\]
	by  a similar manner.
	Thus we conclude that ${\rm ob}_{\infty}(M)=\pi/2$.
\end{proof}

%
%

%
%
%

\begin{remark} \label{rem:rigid0} \upshape
	From the argument in the proof of Theorem \ref{thm:Mongoldt}, we directly obtain
	\[
	\widetilde{\rm ob}_\infty(M) = \pi/2,
	\]
	under the the same hypothesis as Theorem \ref {thm:Mongoldt}, where $\widetilde{\rm ob}_\infty(M)$
	is the  comparison obtuse constant  of $M$ from infinity defined in Section \ref{sec:compa-obtuse}.
	See also \eqref{eq:trivial2}.
\end{remark}

\begin{remark} \label{rem:rigid} \upshape
	Theorem \ref{thm:Mongoldt}
	shows that the estimate $1$-${\rm inj}(M(\infty))\ge \pi/2$ in Theorem \ref{thm:rigid} (1) is sharp.
	%
	It should also be noted that one  cannot expect that 
	$M(\infty)$ has no singular points in  Theorem \ref{thm:rigid} (1), because if one take 
	$N=M\times\mathbb R$, where $M$ is a non-flat open surface as in Theorem \ref{thm:Mongoldt}, 
	then  $\mathrm{ob}_{\infty}(N) = \pi/2$ and $N(\infty)$ is the spherical suspension over $M(\infty)$.
	Note that  $N(\infty)$ has the two singular points at the vertices of the suspension since
	the length of the circle $M(\infty)$ is less than $2\pi$.
\end{remark}

\begin{proof}[Proof of Theorem \ref{thm:rigid}]
	Theorem \ref{thm:rigid} (2) immediately follows from Theorem \ref{thm:ob infty = pi/2} and the splitting theorem.
	Suppose $M$ has no boundary. For $\e_i\to 0$ and $p\in M$, consider the asymptotic limit 
	\begin{align}
	\lim_{i\to\infty} (\e_i M, p) = (M_{\infty}, o),   \label{eq:alimit}
	\end{align}
	where the asymptotic cone $M_{\infty}$ is the Euclidean cone over $M(\infty)$.
	Identify $M(\infty)=M(\infty)\times \{ 1\}\subset M_{\infty}$.
	For every $\xi\in M(\infty)$ and every geodesic $\gamma:[0, \delta]\to M_{\infty}$ from 
	$\xi$, fix any  $0<a<\delta$, and set  $\eta:=\gamma(\delta)$ and $\xi_a=\gamma(a)$. 
	Take sequences $p_i, q_i$ and $x_i\in p_i q_i$ in $\e_i M$ such that 
	\[
	p_i\to \xi,\,\,\,  x_i\to \xi_a, \,\,\, q_i\to \eta,
	\]
	as $i\to\infty$ under the convergence \eqref{eq:alimit}.
	On the interior of a minimal geodesic $p_i x_i$, take points $y_{i,\alpha}$ such that $y_{i,\alpha}\to  x_i$
	as $\alpha\to\infty$. 
	%
	%
	From the assumption, for any sequences $R_i\to\infty$ and $o_i\to 0$, if $\alpha$ is large enough compared to $i$, 
	one can find points $z_i$ with $|x_i,z_i|\ge R_i/\e_i$ and either 
	\[
	%
	\tilde\angle z_ix_i y_{i,\alpha} >\pi - o_i    \,\,\, {\rm  or}   \,\,\, \tilde\angle z_i y_{i,\alpha} x_i>\pi - o_i.
	\]
	Letting  $i\to\infty$, we obtain a geodesic ray emanating from $\xi_a$ either in the direction $\gamma'(a)$
	or  in the  opposite direction $-\gamma'(a)$.
	Then letting $a\to 0$, we conclude that there is a geodesic ray $\sigma$ starting from $\xi$ such that 
	either  $\sigma'(0)=\gamma'(0)$ or else there is the opposite direction $-\gamma'(0)$ and $\sigma'(0)=-\gamma'(0)$.
	Thus we have  $1$-${\rm inj}(M_{\infty})=\infty$.
	
	Now for any direction $v\in\Sigma_{\xi}(M(\infty))\subset \Sigma_{\xi}(M_{\infty})$, there is a geodesic ray 
	$\sigma$ of $M_{\infty}$ starting from $\xi$ either in the direction $v$ or else  in the direction $-v$ if  $-v$ exists.
	For each $t\ge 0$, let $\xi_t:=\uparrow_o^{\sigma(t)}\in M(\infty)$.
	It is easy to see that there is a unique limit $\xi'=\lim_{t\to\infty}\,\xi_t$ and 
	$\angle (\xi,\xi')=\pi/2$.
	Thus $\xi_t$ provides a shortest segment  in $M(\infty)$ from $\xi$ to $\xi'$
	in the direction $v$ or $-v$ if any.
	This shows  $1$-${\rm inj}(M(\infty))\ge \pi/2$ as required. 
\end{proof}

\begin{conjecture} \label{conj:sing}\upshape
	For a compact Alexandrov space  $M$ with curvature bounded below, 
	if ${\rm ob}(M)=\pi/2$, then the double $D(M)$ would have no singular points.
	Similarly, for a noncompact Alexandrov space  $M$ with curvature bounded below, 
	if ${\rm ob_{\infty}}(M)=\pi/2$, then the double $D(M)$ would have no singular points as well
	(compare Section \ref{sec:compa-obtuse}).
\end{conjecture}

\section{Comparison obtuse constants}　\label{sec:compa-obtuse}

Let $M$ be a compact $n$-dimensional Alexandrov space of curvature. 
Let $\kappa \le 0$. 
For $p \neq q \in M$, we set 
\[
\widetilde {\rm ob}_\kappa(p;q) := \sup_{x \in B(p,R_M/2)^c} \tilde \angle_\kappa xpq - \pi/2, 
\]
and 
\[
\widetilde {\rm ob}_\kappa(p,q) := \max \left\{ \widetilde {\rm ob}_\kappa(p;q), \widetilde{\rm ob}_\kappa(q;p) \right\}.
\]
Then, we define the {\it comparison obtuse constant} of $M$ by 
\[
\widetilde {\rm ob}(M) := \liminf_{|p,q|\to 0} \widetilde{\rm ob}_\kappa(p,q). 
\]
Note that this is independent of the choice of $\kappa$. 
A trivial relation 
\begin{equation} \label{eq:trivial1}
\widetilde {\rm ob}(M) \le {\rm ob}(M)
\end{equation}
follows from \eqref{eq:angle vs comp angle}. 
Hence, when $\widetilde {\rm ob}(M)$ is maximal, it is natural to be expected that $M$ has a stronger geometric property compared with the maximal case of $\mathrm{ob}(M)$.  
Indeed, we obtain 
\begin{theorem} \label{thm:rigidity1}
Let $M$ be a compact $n$-dimensional Alexandrov space. 
Suppose $\widetilde{\rm ob} (M) = \pi/2$. 
Then, $D(M)$ has no singular points, that is, for every $p \in M$, we have $\Sigma_p = \mathbb S^{n-1}$ or $\Sigma_p = \mathbb S^{n-1}_+$. 
Here, $\mathbb S_+^{n-1}$ denotes the upper half unit sphere.  
\end{theorem}
\begin{proof}
Let $M$ be as in the assumption. 
When $p \in M$ is a boundary point, we already know that $\Sigma_p = \mathbb S^{n-1}_+$ (See Theorem \ref{thm:maximal}). 
Hence, we assume that $p$ is an interior point. 
Take $\xi, \eta \in \Sigma_p$ with $\angle (\xi, \eta) = \diam(\Sigma_p)$, and  
suppose that $\angle(\xi, \eta) < \pi$. 
Let us take sequences $x_i, y_i \in M$ such that $|p,x_i| = |p,y_i| \to 0$, $\uparrow_p^{x_i} \to \xi$ and $\uparrow_p^{y_i} \to \eta$. 
Since $\widetilde{\mathrm{ob}}(M) = \pi/2$, there exists a point $z_i \in M$ such that one of the following holds:
\begin{enumerate}
\item $\wangle x_i y_i z_i \ge \pi - \epsilon_i$ and $|y_i,z_i| \ge R/2;$
\item $\wangle y_i x_i z_i \ge \pi- \epsilon_i$ and $|x_i, z_i| \ge R/2$,
\end{enumerate}
where $\epsilon_i \to 0$.
By extracting a subsequence and by replacing  $x_i$ and $y_i$ if necessarily, 
we may assume 
(1) holds for all $i$.
Under the convergence of $(|p,x_i|^{-1} M, p)$ to the tangent cone $(T_p M, o_p)$, the sequence of broken geodesics $x_i y_i z_i$ converges to a ray starting from $\xi$ through  $\eta$. 
Now we can take a direction $\zeta \in \Sigma_p$ along the ray satisfying $\angle (\xi, \zeta) > \angle (\xi, \eta)$. 
Since this is a contradiction,  we have $\diam(\Sigma_p) = \pi$. 

By the splitting theorem, $T_p M$ is isometric to the product of the line $\ell$ through  $\xi, \eta$ and the space $T'$ of vectors perpendicular to $\ell$.
Let $\Lambda\subset \Sigma_p$ denote the set of directions tangent to $T'$. 
Then, we have $\diam (\Lambda) = \pi$. 
Indeed, if $\bar \xi, \bar \eta \in \Lambda$ attain the diameter of $\Lambda$, 
then taking sequences $\bar x_i, \bar y_i \to p$ with $|\bar x_i, p| = |\bar y_i, p|$ 
so that $\uparrow_p^{\bar x_i} \to \bar \xi$ and $\uparrow_p^{\bar y_i} \to \bar \eta$, 
we have a point $\bar z_i$ in a way similar to the above argument. 
Then, the limit ray of $\bar x_i \bar y_i \bar z_i$ (or $\bar y_i \bar x_i \bar z_i$)
under the convergence $(|p,\bar x_i|^{-1} M, p) \to (T_p M, o)$,
is contained in $T'$. 
The existence of such a ray enforces that $\diam(\Lambda) = \pi$,
and $T'$ is isometric to a product $\mathbb R\times T''$.
Repeating this argument, finally we obtain that $T_pM$ is isometric to $\mathbb R^n$. 
Therefore, $\Sigma_p = \mathbb S^{n-1}$. 
\end{proof}

For noncompact spaces, a {\it comparison obtuse constant from infinity} $\widetilde{\rm ob}_{\infty}(M)$ is also defined as follows. 
Let $M$ be a noncompact complete $n$-dimensional Alexandrov space of curvature $\ge \kappa'$, where $\kappa' \le 0$. 
For $\kappa \le 0$ and $p \neq q \in M$, we set 
\[
\widetilde{\rm ob}_{\kappa, \infty}(p,q) := \limsup_{x \to \infty} \max \{\tilde\angle_\kappa xpq, \tilde \angle_\kappa xqp \}. 
\]
Then, we define 
\[
\widetilde{\rm ob}_\infty(M) := \liminf_{|p,q| \to 0} \widetilde{\rm ob}_{\kappa, \infty}(p,q). 
\]
Clearly, $\widetilde{\rm ob}_\infty(M)$ is not depending on $\kappa$ and 
\begin{equation} \label{eq:trivial2}
\widetilde{\rm ob}_\infty(M) \le {\rm ob}_\infty(M). 
\end{equation}

By an argument similar to the proof of Theorem \ref{thm:rigidity1}, we have 
\begin{theorem}
Let $M$ be an $n$-dimensional noncompact complete Alexandrov space. 
Assume that $\widetilde{\rm ob}_\infty(M) = \pi/2$.  
Then $D(M)$ has no singular points.
\end{theorem}

\begin{conjecture} \upshape
Theorem \ref{thm:main-cpt1} and Corollary \ref{cor:cpt3} would hold for $\widetilde{\rm ob}(M)$ instead of ${\rm ob}(M)$ when $M$ is compact.
Theorem \ref{thm:main-noncpt} would hold for $\widetilde{\rm ob}_\infty(M)$ instead of ${\rm ob}_\infty(M)$ when $M$ is noncompact.  
\end{conjecture}
Note that from the trivial inequalities \eqref{eq:trivial1} and \eqref{eq:trivial2}, Theorem \ref{thm:main-cpt2} and the last inequality of Theorem \ref{thm:main-noncpt} are ture even for $\widetilde{\rm ob}(M)$ and $\widetilde{\rm ob}_\infty(M)$.

\par\noindent
\section{Comparison $\kappa$-obtuse constants from infinity}　\label{sec:kappa-obtuse}

We conclude the paper with some comments on another definition 
of ``comparison obtuse constant from infinity'' for noncompact spaces which does depend on the 
lower curvature bound.

Let $M$ be a complete noncompact Alexandrov space with 
curvature $\ge \kappa$, and $p\neq q\in M$.
Using our previous definition of 
$\widetilde{\rm ob}_{\kappa,\infty}(p,q)$, 
set 
\[
  \widetilde{\rm ob}_{\kappa,\infty}(M):= \inf_{p\neq q}\,  \widetilde{\rm ob}_{\kappa,\infty}(p,q).
\]
which we call the {\it comparison $\kappa$-obtuse constant of $M$ from infinity}.
Note that
\[
          \widetilde{\rm ob}_{\kappa,\infty }(M) \le \widetilde{\rm ob}_{\infty }(M)\le {\rm ob}_{\infty }(M).
\]
Clearly the $\kappa$-obtuse constant from infinity does depend on the choice of the lower curvature bound $\kappa$, and 
$\widetilde{\rm ob}_{0, \infty}(M)\ge 0$ for $\kappa=0$.
However if  $\kappa <0$, the $\kappa$-obtuse constant from infinity could be negative.
For instance, if $M$ is the domain bounded by an ideal triangle all of whose vertexes are on the ideal boundary of the hyperbolic plane $\mathbb H^2(-1)$.
Then $\widetilde{\rm ob}_{-1,\infty}(M)=-\pi/2$.

This invariant seems interesting in itself.
For instance, we have the following strong rigidity.

\begin{theorem} \label{thm:rigidity}
Let $M$ be a complete noncompact Alexandrov $n$-space with nonnegative curvature 
satisfying $\widetilde{\rm ob}_{0,\infty} (M) = \pi/2$.
If $M$ has no  boundary, then $M$ is isometric to the Euclidean space 
        $\mathbb R^{n}$.
\end{theorem}

\begin{proof}
Take $r_i\to 0$ and consider the pointed Gromov-Hausdorff convergence
$(r_i M, p) \to (M_{\infty}, o)$, where $M_{\infty}$ is the asymptotic cone,
which is isometric to the  the Euclidean cone $K(M(\infty))$ over the ideal boundary $M(\infty)$.
By Theorem \ref{thm:main-noncpt}, $v_{\infty}(M)>0$ and hence $\dim M_{\infty}=n$.
It suffices to show that $M_{\infty}$ is isometric to $\mathbb R^n$.
First we show that $\diam(M(\infty))=\pi$. Suppose  $\diam(M(\infty))<\pi$ and take
$\xi,\eta\in M(\infty)$ with $|\xi,\eta|=\diam(M(\infty))$.
We identify $M(\infty)$ as $M(\infty)\times 1\subset M_{\infty}$,
and take $x_i, y_i\in r_iM$ such that $x_i\to\xi$, $y_i\to\eta$ under the convergence $(r_iM, p) \to (M_{\infty}, o)$.
 From the assumption, we may assume that there is a geodesic ray $\gamma_i$ emanating from $x_i$ through $y_i$.
Passing to a subsequence, we may also  assume that $\gamma_i$ converges to a geodesic ray
$\gamma_{\infty}$ in $M_{\infty}$ emanating from $\xi$ through $\eta$. 
Obviously we can find a point $z$ on $\gamma_{\infty}$ such that the direction $\zeta=\uparrow_o^z$ 
satisfies $|\xi,\zeta|>|\xi,\eta|$. Since this is a contradiction, we have 
      $\diam(M(\infty))=\pi$.
By the splitting theorem,  $M_{\infty}$ is isometric to a product $M_{\infty}'\times\mathbb R$.
Repeating the argument to  $M_{\infty}'$, we see that  $M_{\infty}'$ is isometric to a product $M_{\infty}''\times\mathbb R$.
In this way, we conclude that  $M_{\infty}$ is isometric to $\mathbb R^n$.
\end{proof}

\end{document}